\documentclass[preprint,12pt]{elsarticle}

\usepackage{amssymb,amsmath, amsthm, amsfonts, amsbsy, amsbsy, latexsym, color, epsfig, graphicx}

\newtheorem{Thm}{Theorem}[section]

\newtheorem{lemma}[Thm]{Lemma}

\newtheorem{proposition}[Thm]{Proposition}
\newtheorem{definition}[Thm]{Definition}

\newtheorem{theorem}[Thm]{Theorem}

\newcommand{\bitem}{\begin{itemize}}
\newcommand{\eitem}{\end{itemize}}
\newcommand{\benum}{\begin{enumerate}}
\newcommand{\eenum}{\end{enumerate}}
\newcommand{\beq}{\begin{equation}}
\newcommand{\eeq}{\end{equation}}
\newcommand{\ip}[2]{\langle#1,#2\rangle}

\newcommand{\norm}[1]{\|#1\|}

\newcommand{\Sp}{\mbox{supp}}
\newcommand{\intt}{\mbox{int}}
  \newcommand{\R}{\mathbb{R}}

 \newcommand{\Z}{\mathbb{Z}}

\DeclareMathOperator*{\supp}{supp}

\def\ZZ{\mathbb{Z}}
\def\RR{\mathbb{R}}
\def\ZZ{\mathbb{Z}}

\def\eps{\varepsilon}

\newcommand{\bZ}{{\mathbb Z}}
\newcommand{\bR}{{\mathbb R}}

\def\cC{{\mathcal{C}}}

\def\cE{{\mathcal{E}}}
\def\cQ{{\mathcal{Q}}}
\def\cR{{\mathcal{R}}}

\def\cS{{\mathcal{S}}}

\def\cSH{{\mathcal{S}\mathcal{H}}}

\newcommand{\wqq}[1]{#1}

\newcommand{\wq}[1]{#1}
\newcommand{\gkk}[1]{#1}
\newcommand{\wqqq}[1]{#1}
\newcommand{\wql}[1]{#1}
\newcommand{\wang}[1]{{{#1}}}

\begin{document}

\begin{frontmatter}



\title{Compactly Supported Shearlets are\\ Optimally Sparse}


\author[GK]{Gitta Kutyniok\corref{cor1}\fnref{fn1}}
\author[GK]{Wang-Q Lim\fnref{fn1}}

\address[GK]{Institute of Mathematics, University of Osnabr\"uck, 49069 Osnabr\"uck, Germany}

\cortext[cor1]{Corresponding author}

\fntext[fn1]{G.K. and W.-Q L. would like to thank Wolfgang Dahmen,
David Donoho, Chunyan Huang, Demetrio Labate, Christoph Schwab, and
Gerrit Welper for various discussions on related topics. G.K. and
W.-Q L. acknowledge support from DFG Grant SPP-1324, KU 1446/13. G.K. was
also partially supported by DFG Grant, KU 1446/14.}

\begin{abstract}
Cartoon-like images, i.e., $C^2$ functions which are smooth apart from
a $C^2$ discontinuity curve, have by now become a standard model for measuring sparse (non-linear)
approximation properties of directional representation systems. It was already shown that
curvelets, contourlets, as well as shearlets do exhibit (almost) optimally sparse approximations
within this model. However, all those results are only applicable to band-limited
generators, whereas, in particular, spatially compactly supported generators are of
uttermost importance for applications.

In this paper, we now present the first complete proof of (almost) optimally sparse approximations
of cartoon-like images by using a particular class of directional representation systems, which
indeed consists of compactly supported elements. This class will be chosen as a subset of
shearlet frames -- not necessarily required to be tight -- with shearlet generators having compact
support and satisfying some weak moment conditions.
\end{abstract}








\end{frontmatter}



\section{Introduction}

{In computer vision, edges were detected as those features} governing an image
while separating smooth regions in between. About 10 years ago, mathematicians
started to design models of images incorporating those findings aiming at
designing representation systems which -- in such a model -- are capable of resolving
edges in an optimally sparse way. However, customarily, at that time an image
was viewed as an element of a compact subset of $L_p$ characterized by a given
Besov regularity with the Kolmogorov entropy of such sets identifying lower
bounds for the distortion rates of encoding-decoding pairs in this model.
Although wavelets could be shown to behave optimally \cite{CDDD} as an
encoding methodology, Besov models are clearly deficient since edges are not
adequately captured. This initiated the introduction of a different model,
called cartoon-like model (see \cite{Don99,WRHB02,CD04}), which revealed the
suboptimal treatment of edges by wavelets.

The introduction of tight curvelet frames in 2004 by Cand\'{e}s and Donoho
\cite{CD04}, which provably provide (almost) optimally sparse approximations
within such a cartoon-like model can be considered a milestone in applied harmonic
analysis. One year later, contourlets were introduced by Do and Vetterli \cite{DV05}
which similarly derived (almost) optimal approximation rates. In the same year,
{\em shearlets} were developed by Labate, Weiss, and the authors in \cite{LLKW05} as
the first directional representation system with allows a unified treatment
of the continuum and digital world similar to wavelets, while also providing
(almost) optimally sparse approximations within such a cartoon-like model \cite{GL07}.

In most applications, spatial localization of the analyzing elements of an
encoding system is of
uttermost importance both for a precise detection of geometric features as well as for a fast
decomposition algorithm. However, none of the previously mentioned results
cover this situation. In fact, the proofs which were provided do by no means
extend to this crucial setting.

\medskip

In this paper, we now present the first complete proof of (almost) optimally sparse approximations
of cartoon-like images by using a particular class of directional representation systems, which
indeed consist of compactly supported elements. This class will be chosen as a subset of
shearlet frames -- not necessarily required to be tight -- with shearlet generators having compact
support and satisfying some weak moment conditions. Interestingly, our proof is very different from
all previous ones caused by the extensive exploration of the compact support of
the shearlet generators and the lack of directional vanishing moments.


\subsection{A Suitable Model for Images: Cartoon-Like Images}

Intuitively, cartoons are smooth image parts separated from other areas by an edge.
After a series of initial models \cite{Don99,WRHB02}, the first complete model of
cartoons has been introduced in \cite{CD04}, and this is what we intend to use also
here. The basic idea is to choose a closed boundary curve and then fill the interior
and exterior part with $C^2$ functions (see Figure \ref{fig:Cartoon}).
\begin{figure}[h]
\begin{center}
\vspace*{0.3cm}
\includegraphics[height=1.25in]{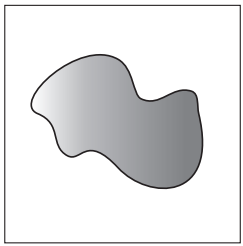}
\vspace*{-0.3cm}
\end{center}
\caption{Example of a cartoon-like image.}
\label{figure}
\label{fig:Cartoon}
\end{figure}

Let us now be more precise, and introduce $STAR^2(\nu)$, a class of indicator functions of
sets $B$ with $C^2$ boundaries $\partial B$ and curvature bounded by $\nu$, as well as
$\cE^2(\nu)$, a class of cartoon-like images. For this, in polar coordinates, we let $\rho(\theta)
\rightarrow [0,1]$ be a radius function and define the set $B$ by
\[
B = \{x \in \RR^2 : |x| \le \rho(\theta), x = (|x|,\theta) \mbox{ in polar coordinates}\}.
\]
In particular, the boundary $\partial B$ of $B$ is given by the curve
\begin{equation}\label{eq:curve}
\beta(\theta) = \begin{pmatrix} \rho(\theta)\cos(\theta) \\
\rho(\theta)\sin(\theta)\end{pmatrix},
\end{equation}
 and the class of boundaries of interest to us are defined by
\begin{equation}\label{eq:curvebound}
\sup|\rho^{''}(\theta)| \leq \nu, \quad \rho \leq \rho_0 < 1.
\end{equation}
The following definition now introduces the notions $STAR^2(\nu)$ and $\cE^2(\nu)$ from \cite{CD04}.

\begin{definition}
For $\nu > 0$, the set $STAR^2(\nu)$ is defined to be the set of all $B \subset [0,1]^2$ such that $B$
is a translate of a set obeying \eqref{eq:curve} and \eqref{eq:curvebound}. Further, $\cE^2(\nu)$
denotes the set of functions $f$ on $\RR^2$ with compact support in $[0,1]^2$ of the form
\[
f = f_0 + f_1 \chi_{B},
\]
where $f_0,f_1 \in C^2(\RR^2)$ with compact support in $[0,1]^2$, $B \in STAR^2(\nu)$, and $\|f\|_{C^2} = \sum_{|\alpha| \leq 2}
\|D^{\alpha}f\|_{\infty} \leq 1.$
\end{definition}


\subsection{Optimal Sparsity of a Directional Representation System}
\label{subsec:optsparse}

The `quality' of the performance of a (directional) representation system with respect to
cartoon-like images is typically measured by taking a non-linear approximation viewpoint.
More precisely, given a cartoon-like image $f \in \cE^2(\nu)$ and a (directional)
representation system $(\sigma_i)_{i \in I}$ which forms an orthonormal basis, the chosen measure
is the asymptotic behavior of the best $N$-term (non-linear) approximation error in $L^2$ norm
in the number of terms $N$, i.e.,
\[
\norm{f-f_N}_2^2 = \Big\|f - \sum_{i \in I_N} \ip{f}{\sigma_i}\sigma_i\Big\|_2^2 \quad \mbox{as } N \to \infty,
\]
where $(\ip{f}{\sigma_i})_{i \in I_N}$ are the $N$ largest coefficients $\ip{f}{\sigma_i}$
in magnitude. Wavelet bases exhibit the approximation rate
\[
\norm{f-f_N}_2^2 \le C \cdot N^{-1} \quad \mbox{as } N \to \infty.
\]
However, Donoho proved in \cite{Don01} that the optimal rate which can be achieved under some
restrictions on the representation system as well as on the selection procedure of
the approximating coefficients is
\[
\norm{f-f_N}_2^2 \le C \cdot N^{-2} \quad \mbox{as } N \to \infty.
\]
It was a breakthrough in 2004, when Cand\'{e}s and Donoho introduced the tight curvelet
frame in \cite{CD04} and proved  that this system indeed does satisfy
\[
\norm{f-f_N}_2^2 \le C \cdot N^{-2} \cdot (\log N)^{3} \quad \mbox{as } N \to \infty,
\]
where again the approximation $f_N$ was generated by the $N$ largest coefficients in magnitude.
Although the optimal rate is not completely achieved, the $\log$-factor is typically considered negligible
compared to the $N^{-2}$-factor, wherefore the term `almost optimal' has been adopted into the
language. This result is even more surprising taking into account that in case of a tight frame
the approximation by the $N$ largest coefficients in magnitude does not even always yield the
{\em best} $N$-term approximation.


\subsection{(Compactly Supported) Shearlet Systems}

The directional representation system of {\em shearlets} has recently emerged and rapidly gained
attention due to the fact that -- in contrast to other proposed directional representation systems --
shearlets provide a unified treatment of the continuum and digital world similar to wavelets.
We refer to, e.g., \cite{GKL06,KL09} for the continuum theory, \cite{DKS08,ELL08,Lim09} for the digital
theory, and \cite{GLL09,DK10} for recent applications.

Shearlets are scaled according to a parabolic scaling law encoded in the {\em parabolic scaling matrices}
$A_{2^j}$ or $\tilde{A}_{2^j}$, $j \in \ZZ$, and exhibit directionality by parameterizing slope encoded
in the {\em shear matrices} $S_k$, $k \in \ZZ$, defined by
\[
A_{2^j} =
\begin{pmatrix}
  2^j & 0\\ 0 & 2^{j/2}
\end{pmatrix}
\qquad\mbox{or}\qquad
\tilde{A}_{2^j} =
\begin{pmatrix}
  2^{j/2} & 0\\ 0 & 2^j
\end{pmatrix}
\]
and
\[
S_k = \begin{pmatrix}
  1 & \wql{k} \\ 0 & 1
\end{pmatrix},
\]
respectively.

To ensure an (almost) equal treatment of the different slopes, which is
evidently of significant importance for practical applications, we partition the frequency
plane into the following four cones $\cC_1$ -- $\cC_4$:
\[
\cC_\iota = \left\{ \begin{array}{rcl}
\{(\xi_1,\xi_2) \in \bR^2 : \xi_1 \ge 1,\, |\xi_2/\xi_1| \le 1\} & : & \iota = 1,\\
\{(\xi_1,\xi_2) \in \bR^2 : \xi_2 \ge 1,\, |\xi_1/\xi_2| \le 1\} & : & \iota = 2,\\
\{(\xi_1,\xi_2) \in \bR^2 : \xi_1 \le -1,\, |\xi_2/\xi_1| \le 1\} & : & \iota = 3,\\
\{(\xi_1,\xi_2) \in \bR^2 : \xi_2 \le -1,\, |\xi_1/\xi_2| \le 1\} & : & \iota = 4,
\end{array}
\right.
\]
and a centered rectangle
\[
\cR = \{(\xi_1,\xi_2) \in \bR^2 : \|(\xi_1,\xi_2)\|_\infty < 1\}.
\]
For an illustration, we refer to Figure \ref{fig:shearlet}(a).

\begin{figure}[ht]
\begin{center}
\includegraphics[height=1.4in]{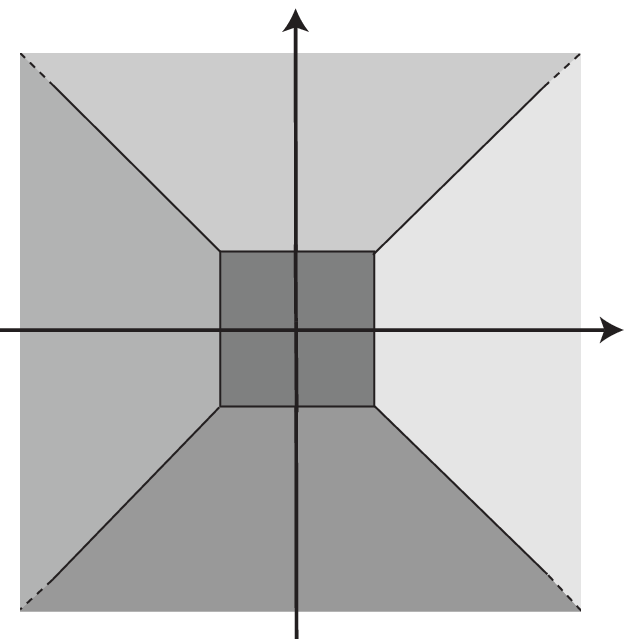}
\put(-33,58){\footnotesize{$\cC=\cC_1$}}
\put(-70,80){\footnotesize{$\cC_2$}}
\put(-88,30){\footnotesize{$\cC_3$}}
\put(-50,52){\footnotesize{$\cR$}}
\put(-45,15){\footnotesize{$\cC_4$}}
\put(-60,-17){(a)}
\hspace*{4cm}
\includegraphics[height=1.4in]{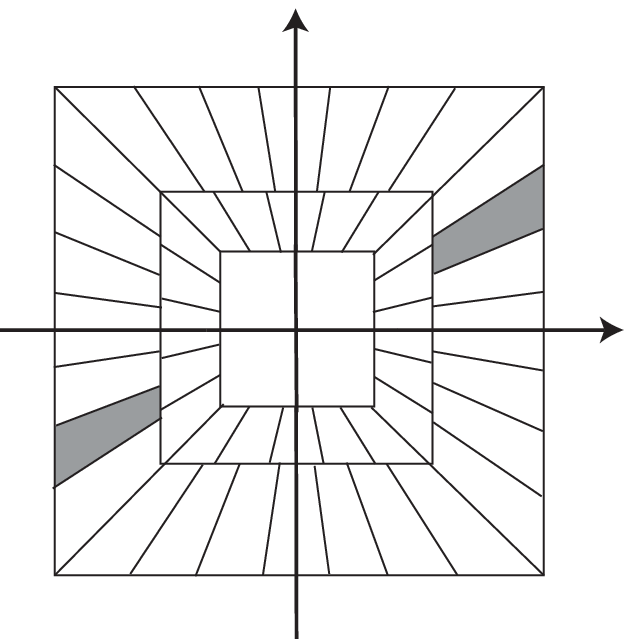}
\put(-60,-17){(b)}
\end{center}
\caption{(a) The cones $\cC_1$ -- $\cC_4$ and the centered rectangle $\cR$ in frequency domain.
(b) The tiling of the frequency domain induced by a (cone-adapted) shearlet system.}
\label{fig:shearlet}
\end{figure}

The rectangle $\cR$ corresponds to the low frequency content of a signal and is customarily
represented by translations of some scaling function. Anisotropy comes into play when
encoding the high frequency content of a signal which corresponds to the cones $\cC_1$ -- $\cC_4$,
where the cones $\cC_1$ and $\cC_3$ as well as $\cC_2$ and $\cC_4$ are treated separately as can be seen
in the following

\begin{definition}
\label{defi:discreteshearlets}
For some sampling constant $c > 0$, the {\em (cone-adapted) shearlet system} $\cSH(c;\phi,\psi,\tilde{\psi})$ generated by
a {\em scaling function} $\phi \in L^2(\mathbb{R}^2)$ and {\em shearlets} $\psi, \tilde{\psi} \in L^2(\mathbb{R}^2)$ is
defined by
\[
\cSH(c;\phi,\psi,\tilde{\psi}) = \Phi(c;\phi) \cup \Psi(c;\psi) \cup \tilde{\Psi}(c;\tilde{\psi}),
\]
where
\[
\Phi(c;\phi) = \{\phi_m  = \phi(\cdot-cm) : m \in \bZ^2\},
\]
\[
\Psi(c;\psi) = \{\psi_{j,k,m} =  2^{3j/4} {\psi}({S}_{k} {A}_{2^j}\cdot-cm) :
j \ge 0, |k| \le \lceil 2^{j/2} \rceil, m \in \bZ^2 \},
\]
and
\[
\tilde{\Psi}(c;\tilde{\psi}) = \{\tilde{\psi}_{j,k,m} =  2^{3j/4} \tilde{\psi}(S^T_{k} \tilde{A}_{2^j}\cdot-cm) :
j \ge 0, |k| \le \lceil 2^{j/2} \rceil, m \in \bZ^2 \}.
\]
\end{definition}
The reader should keep in mind that although not indicated by the notation, the functions $\phi_m$, $\psi_{j,k,m}$, and
$\tilde{\psi}_{j,k,m}$ all depend on the sampling constant $c$. For the sake of brevity, we will often write $\psi_\lambda$ and $\tilde{\psi}_{\lambda}$,
where $\lambda = (j,k,m)$ index scale, shear, and position. For later use, we further let $\Lambda_j$ be the indexing sets
of shearlets in $\Psi(c;\psi)$ and $\tilde{\Psi}(c;\tilde{\psi})$ at scale $j$,
respectively, i.e.,
\[
\Psi(c;\psi) = \{\psi_{\lambda} : \lambda \in \Lambda_j, j=0, \ldots, \infty\}
\]
and
\[
\tilde{\Psi}(c;\tilde{\psi}) = \{\tilde{\psi}_{\lambda} : \lambda \in \Lambda_j, j=0, \ldots, \infty\}.
\]
Finally, we define
\[
\Lambda = \bigcup_{j=0}^\infty \Lambda_j.
\]

The tiling of frequency domain induced by $\cSH(c;\phi,\psi,\tilde{\psi})$ is illustrated in Figure \ref{fig:shearlet}(b).
From this illustration, the anisotropic footprints of shearlets contained in $\Psi(c;\psi)$ and $\tilde{\Psi}(c;\tilde{\psi})$
can clearly be seen. However, the reader should notice that the tiling indicated here is based on the {\em essential support} and not
the exact support of the analyzing elements, since our focus will be on shearlet systems associated with spatially compactly
supported generators.
The corresponding anisotropic footprints of shearlets {\em in spatial domain} are of size $2^{-j/2}$ times
$2^{-j}$. A beautiful intuitive extensive explanation of why it is conceivable that such a system -- based on parabolic scaling --
exhibits optimal sparse approximation of cartoon-like images, is provided in \cite{CD04}, and we would like
to refer the reader to this paper. The main idea is to count the number of shearlets intersecting the
discontinuity curves, which is `small' compared to the number of such wavelets, due to their anisotropic footprints.

Certainly, we naturally ask the question when $\cSH(c;\phi,\psi,\tilde{\psi})$ does form a frame for $L^2(\RR^2)$.
The wavelet literature provides various necessary and sufficient conditions for $\Phi(c;\phi)$ to
form a frame for $L^2(\{f \in L^2(\RR^2) : \Sp(\wq{\hat{f}}) \subseteq \cR\})$, also when $\phi$ is compactly
supported in spatial domain. Although not that well-studied as wavelets yet, several answers are also
known for the question when $\Psi(c;\psi)$ forms a frame for
\[
L^2(\{f \in L^2(\RR^2) : \Sp(\wq{\hat f}) \subseteq \cC_1 \cup \cC_3\}),
\]
and we refer to results in \cite{GKL06,KL07,DKST09,KKL10b}.
Since $\Psi(c;\phi)$ and $\tilde{\Psi}(c;\tilde{\psi})$ are linked by a simple rotation of $90^o$, these results
immediately provide conditions for $\tilde{\Psi}(c;\tilde{\psi})$ to constitute a frame for
\[
L^2(\{f \in L^2(\RR^2) : \Sp(\wq{\hat f}) \subseteq \cC_2 \cup \cC_4\}).
\]
Very recent results in \cite{KKL10a}
even focus specifically on the case of spatially compactly supported shearlets --  of uttermost
importance for applications due to their superior localization. For instance, in \cite{KKL10a},
the following special class of compactly supported shearlet frames for $L^2(\bR^2)$ was constructed:
The generating shearlets $\psi$  and $\tilde{\psi}$ were chosen separable, i.e., of the form
$\psi_1(x_1)\cdot\psi_2(x_2)$  and $\psi_1(x_2)\cdot\psi_2(x_1)$, respectively, where $\psi_1$ is a wavelet and $\psi_2$ is
a scaling function both associated with some carefully chosen low pass filter. Intriguingly,
our main result in this paper (Theorem \ref{theo:main}) proves as a corollary that this
class of compactly supported shearlet frames provides (almost) optimally sparse approximations
of cartoon-like images. We refer to \cite{KKL10a} for the precise statement.

Combining those thoughts, we can attest that frame properties of the system $\cSH(c;\phi,\psi,\tilde{\psi})$
including spatially compactly supported generators are already quite well studied.
We however wish to mention that there is a trade-off between {\em compact support} of the shearlet generators, {\em tightness} of the associated frame,
and {\em separability} of the shearlet generators. The known constructions of tight shearlet frames do not use separable generators, and
these constructions can be shown to {\em not} be applicable to compactly supported generators. Tightness is difficult to obtain while
allowing for compactly supported generators, but we can gain separability, hence fast algorithmic realizations. On the other hand, when allowing
non-compactly supported generators, tightness is possible, but separability seems to be out of reach, which
makes fast algorithmic realizations very difficult.


\subsection{Optimally Sparse Approximation of Cartoon-Like Images by Shearlets}

The concept of optimally sparse approximation of cartoon-like images of general (directional)
representation systems was already discussed in Section \ref{subsec:optsparse}.
However, the attentive reader will have realized that only the situation of
tight frames was studied whereas here we need to consider sparse
approximations by arbitrary frames. Hence this situation deserves a careful
commenting.

Let $\cSH(c;\phi,\psi,\tilde{\psi})$  be a shearlet frame for $L^2(\RR^2)$, which for
illustrative purposes for a moment we denote by $\cSH(c;\phi,\psi,\tilde{\psi}) = (\sigma_i)_{i \in I}$,
say. Is it well-known that a frame is associated with a canonical dual frame, which in this
case we want to call $(\tilde{\sigma}_i)_{i \in I}$. Then we define the $N$-term approximation $f_N$
of a cartoon-like image $f \in \cE^2(\nu)$ by the frame $\cSH(c;\phi,\psi,\tilde{\psi})$ to be
\[
f_N = \sum_{i \in I_N} \ip{f}{\sigma_i}\tilde{\sigma}_i,
\]
where $(\ip{f}{\sigma_i})_{i \in I_N}$ are the $N$ largest coefficients $\ip{f}{\sigma_i}$
in magnitude. As in the tight frame case, this procedure does not always yield the {\em best}
$N$-term approximation, but surprisingly even with this `crude' selection procedure -- in
the situation of spatially compactly supported generators -- we can prove an (almost) optimally
sparse approximation rate as our main result shows.

\begin{theorem}\label{theo:main}
Let $c > 0$, and let $\phi, \psi, \tilde{\psi} \in L^2(\RR^2)$ be compactly supported. Suppose
that, in addition, for all $\xi = (\xi_1,\xi_2) \in \RR^2$, the shearlet $\psi$ satisfies\\[-1.25ex]
\bitem
\item[(i)] $|\hat\psi(\xi)| \le C_1 \cdot \min(1,|\xi_1|^{\alpha}) \cdot \min(1,|\xi_1|^{-\gamma}) \cdot \min(1,|\xi_2|^{-\gamma})$ and\\[-0.1ex]
\item[(ii)] $\left|\frac{\partial}{\partial \xi_2}\hat \psi(\xi)\right|
\le  \wqqq{|h(\xi_1)|} \cdot \left(1+\frac{|\xi_2|}{|\xi_1|}\right)^{-\gamma}$,\\[-0.2ex]
\eitem
where $\alpha > 5$, $\gamma \ge 4$, $h \in L^1(\bR)$, and $C_1$ is a constant, and suppose that
the shearlet $\tilde{\psi}$ satisfies (i) and (ii) with the roles of $\xi_1$ and $\xi_2$ reversed.
Further, suppose that $\cSH(c;\phi,\psi,\tilde{\psi})$ forms a
 frame for $L^2(\RR^2)$.

Then, for any $\nu > 0$, the shearlet frame $\cSH(c;\phi,\psi,\tilde{\psi})$ provides (almost)
optimally sparse approximations of functions $f \in \cE^2(\nu)$, i.e., there exists some $C > 0$ such
that
\[
\|f-f_N\|_2^2 \leq C\cdot N^{-2} \cdot {(\log{N})}^3 \qquad \text{as } N \rightarrow \infty,
\]
where $f_N$ is the nonlinear N-term approximation obtained by choosing the N largest shearlet coefficients
of $f$.
\end{theorem}

Condition (i) can be interpreted as both a condition ensuring (almost) separable behavior
as well as a first order moment condition along the horizontal axis, hence enforcing directional selectivity.
This condition ensures that the support of shearlets in frequency domain is essentially of
the form indicated in Figure \ref{fig:shearlet}(b). Condition~(ii) (together with (i)) is a weak version of a directional vanishing
moment condition\footnote{For the precise definition of directional
  vanishing moments, we refer to \cite{DV05}. }, which is crucial for
having fast decay of the shearlet coefficients when the corresponding
shearlet intersects the discontinuity curve. Conditions~(i) and (ii)
are rather mild conditions on the generators. To compare with the
optimality result for band-limited generators we wish to point out
that conditions~(i) and (ii) are obviously satisfied for band-limited
generators.

Notice also that, intriguingly, the -- the `true' optimality destroying --  $\log$-factor has the
{\em same} exponent as in the curvelet-, contourlet-, and shearlet-result on (almost) optimally
sparse approximation.


\subsection{Prior Work and Our Contribution}

In 2004, Cand\'{e}s and Donoho \cite{CD04} achieved a breakthrough when introducing
tight curvelet frames, which provide (almost) optimally sparse approximations of cartoon-like
images (functions in $\cE^2(\nu)$). The main outline of their proof is to break $[0,1]^2$ into smaller
cubes and then separately analyze the
curvelet coefficients essentially centered in the smooth part of the model and those essentially
centered on the discontinuity curve. For both sets of coefficients their weak-$\ell_{2/3}$ norm
is estimated; the estimate for the `non-smooth part' also requiring the usage of the Radon transform.

A year later, Do and Vetterli \cite{DV05} introduced contourlets and proved similar sparsity results
for those. However, although their work includes contourlets with compact support, their construction
is fully based on discrete filter banks so that directional selectivity is problematic. Because of
this fact, infinite directional vanishing moments had to be artificially imposed in order to achieve
(almost) optimal sparsity. However, this is impossible for any function with compact support to
satisfy. Hence, similar to curvelets, optimal sparsity is only proven for {\em band-limited}
contourlets.

In 2005, shearlets were introduced as the first directional representation system ensuring a unified
treatment of the continuum and digital world by Labate, Weiss, and the authors in \cite{LLKW05}.
One year later, Labate and Guo proved (almost) optimally sparse approximations of cartoon-like images
for the at that time customarily utilized shearlet frames \cite{GL07}, which are band-limited
such as curvelets. The proof the authors provided follows the proof in \cite{CD04} very closely
step by step.

Concluding, although those pioneering studies deserve all our credit, these results are far from
including the important class of directional representation systems consisting of compactly
supported functions.

\medskip

The main contribution of this paper is to provide the first complete proof of (almost)
optimally sparse approximations of cartoon-like images using a directional representation system
consisting of compactly supported functions. Our proof is indeed very different from all previous
ones caused by the necessary extensive exploration of the compact support of the shearlet generators,
the only similarity being the breaking of $[0,1]^2$ into smaller cubes and the separate
consideration of shearlet coefficients now being {\em exactly contained} -- in contrast to
being {\em essentially contained} for all
other systems --  in the smooth part and those which intersect the discontinuity curve.
Previous results all require moment conditions along the direction of the discontinuity
curve -- thereby requiring vanishing moments along infinitely many directions asymptotically
in scale --, which
is trivially satisfied for band-limited generators. Intriguingly, a weaker version of
directional vanishing moments, even only in one direction and the shearing taking care of the
remaining directions, is sufficient for our analysis.




\subsection{Outline}

In Section \ref{sec:architecture}, we present the overall structure of the proof, the
results of the analysis of shearlet coefficients being contained in the smooth part and
those which intersect the discontinuity curve, and -- based on these results -- state the proof
of Theorem \ref{theo:main}. The proofs of the results on the behavior of shearlet coefficients
in the smooth and non-smooth part are then carried out in Sections \ref{sec:smoothpart} and
\ref{sec:discontinuity}, respectively.


\section{Architecture of the Proof of Theorem \ref{theo:main}}
\label{sec:architecture}

We now detail the general structure of the proof of Theorem \ref{theo:main}, starting by
introducing useful notions and explaining the blocking into smaller boxes and splitting
into the smooth and non-smooth part. Then the main results concerning the analysis of
shearlet coefficients being entirely contained in the smooth part and those intersecting
the discontinuity curve will be presented followed by the proof of Theorem \ref{theo:main}
based on those.


\subsection{General Organization}

Let now $\cSH(c;\phi,\psi,\tilde{\psi})$ satisfy the hypotheses of Theorem \ref{theo:main}, and
let $f \in \cE^2(\nu)$. Further, we let $A$ denote the lower frame bound of $\cSH(c;\phi,\psi,\tilde{\psi})$.

We first observe that, without loss of generality, we might assume
the scaling index $j$ to be sufficiently large, since $f$ as well as all frame elements in the shearlet
frame $\cSH(c;\phi,\psi,\tilde{\psi})$ are compactly supported in spatial domain, hence a finite
number does not contribute to the asymptotic estimate we aim for. In particular, this means that
we do not need to consider frame elements from $\Phi(c;\phi)$. Also, we are allowed to restrict our
analysis to shearlets $\psi_{j,k,m}$, since the frame elements $\widetilde{\psi}_{j,k,m}$ can be
handled in a similar way.

Our main concern will be to derive appropriate estimates for the shearlet coefficients
$\{\ip{f}{\psi_{\lambda}} : \lambda \in \Lambda\}$ of $f$. Letting $|\theta(f)|_n$ denote
the $n$th largest shearlet coefficient $\ip{f}{\psi_{\lambda}}$ in absolute value and exploring
the frame property of  $\cSH(c;\phi,\psi,\tilde{\psi})$, we conclude that
\[
\|f-f_N\|_2^2 \leq \frac{1}{A}\sum_{n>N} |\theta(f)|_n^2,
\]
for any positive integer $N$. Thus, for the proof of Theorem \ref{theo:main}, it suffices to show that
\begin{equation}\label{eq:upper}
\sum_{n>N} |\theta(f)|_n^2 \leq C \cdot N^{-2} \cdot  {(\log{N})}^3 \qquad \text{as } N \rightarrow \infty.
\end{equation}

To derive the anticipated estimate in \eqref{eq:upper}, for any shearlet
$\psi_{\lambda}$, we will study two separate cases:
\bitem
\item {\em Case 1}. The compact support of the shearlet $\psi_{\lambda}$ does not intersect the boundary of the set $B$, i.e.,
$\wang{\intt(\Sp(\psi_{\lambda}))} \cap \partial B = \emptyset.$
\item {\em Case 2}. The compact support of the shearlet $\psi_{\lambda}$ does intersect the boundary of the set $B$, i.e.,
$\wang{\intt(\Sp(\psi_{\lambda}))} \cap \partial B \ne \emptyset.$
\eitem
Notice that this exact distinction is only possible due to the spatial compact support of all shearlets
in the shearlet frame.

In the sequel -- since we are concerned with an asymptotic estimate -- for simplicity we will often simply use
$C$ as a constant although it might differ for each estimate. Also all the results in the sequel are independent
on the sampling constant $c>0$, wherefore we now fix it once and for all.


\subsection{The Smooth and the Non-Smooth Part of a Cartoon-Like Image}

To illustrate which conditions on $\psi$ required by Theorem \ref{theo:main} are utilized
for the decay estimates of the different cases, in this section we do not make any initial
assumptions on $\psi$.

Let us start with the smooth part, which is the easier one to handle. Dealing with this part
allows us to consider some $g \in C^2(\RR^2)$ with compact support in $[0,1]^2$ and estimate its shearlet coefficients. This
is done in the following proposition. Notice that the hypothesis on $\psi$ of the following result
is implied by condition (i) in Theorem \ref{theo:main}.

\begin{proposition}
\label{prop:main1}
Let $g \in C^2(\RR^2)$ with compact support in $[0,1]^2$, and let $\psi \in L^2(\RR^2)$ be compactly supported and satisfy
\[
|\hat\psi(\xi)| \le C_1 \cdot \min(1,|\xi_1|^{\alpha}) \cdot \min(1,|\xi_1|^{-\gamma}) \cdot \min(1,|\xi_2|^{-\gamma})
\mbox{ for all }\xi = (\xi_1,\xi_2) \in \RR^2 \hspace*{-0.17cm},
\]
where $\gamma>3$, $\alpha > \gamma+2$, and $C_1$ is a constant. Then, there exists
some $C>0$ such that
\[
\sum_{n>N} |\theta(g)|_n^2 \le C \cdot N^{-2} \qquad \text{as } N \rightarrow \infty.
\]
\end{proposition}

Thus, in this case, optimal sparsity is achieved. The proof of this proposition is given in
Section \ref{sec:smoothpart}.

Next, we turn our attention to the non-smooth part, in particular, to estimating those shearlet
coefficients whose spatial support intersects the discontinuity curve. For this, we first
need to introduce some new notations. For any scale $j \ge 0$ and any grid point $p \in \ZZ^2$, we
let $\cQ_{j,p}$ denote the dyadic cube defined by
\[
\cQ_{j,p} = [-2^{-j/2},2^{-j/2}]^2+2^{-j/2}p.
\]
Further, let $\cQ_j$ be the collection of those dyadic cubes $\cQ_{j,p}$ whose interior, in the following denoted by
$\intt(\cQ_{j,p})$, intersects $\partial B$, i.e.,
\[
\cQ_j = \{\cQ_{j,p}: \intt(\cQ_{j,p}) \cap \partial B \ne \emptyset, p \in \ZZ^2\}.
\]
Of interest to us is also the set of shearlet indices, which are associated with shearlets intersecting
the discontinuity curve inside some $\cQ_{j,p}  \in \cQ_{j}$, i.e., for $j\ge0$ and $p \in \bZ^2$ with $\cQ_{j,p} \in \cQ_{j}$,
we will consider the index set
\[
{\Lambda_{j,p}} = \{\lambda \in \Lambda_j : \wqq{\intt(\Sp(\psi_{\lambda})) \cap \intt(\cQ_{j,p})} \cap \partial B \ne \emptyset\}.
\]
Finally, for $j \ge 0$, $p \in \ZZ^2$, and $0 < \eps < 1$, we define $\Lambda_{j,p}(\eps)$ to be \wq{the index set of} shearlets
$\psi_\lambda$, $\lambda \in {\Lambda_{j,p}}$, such that the magnitude of the corresponding shearlet
coefficient $\langle f,\psi_\lambda \rangle$ is larger than $\eps$ and the support of $\psi_\lambda$
intersects $\cQ_{j,p}$ at the $j$th scale, i.e.,
\[
{\Lambda_{j,p}(\eps)} = \{\lambda \in {\Lambda_{j,p}} : |\langle f,\psi_{\lambda}\rangle| > \eps\},
\]
and we define {$\Lambda(\eps)$} to be the \wq{index set for} shearlets so that $|\langle f,\psi_{\lambda} \rangle| > \eps$
across all scales $j$, i.e.,
\[
\Lambda(\eps) = \bigcup_{j,p} \Lambda_{j,p}(\eps).
\]
The expert reader will have noticed that in contrast to the proofs in \cite{CD04} and \cite{GL07},
which also split the domain into smaller scale boxes, we do not apply a weight function to
obtain a smooth partition of unity. In our case, this is not necessary due to the spatial compact
support of the frame elements.

As mentioned at the beginning of this section, we may assume that $j$ is sufficiently large. Given
some scale $j \ge 0$ and position $p \in \ZZ^2$ for which the associated cube $\cQ_{j,p}$
satisfies $\cQ_{j,p} \in \cQ_j$. Then the set
\[
\cS_{j,p}= \bigcup_{\lambda \in \Lambda_{j,p}} \Sp(\psi_\lambda)
\]
is contained in a cubic window of size $C\cdot 2^{-j/2}$ by $C\cdot 2^{-j/2}$, hence is of asymptotically
the same size as $\cQ_{j,p}$. By smoothness assumption on the discontinuity curve $\partial B$, the edge curve can be parameterized by either
$(x_1,E(x_1))$ or $(x_2,E(x_2))$ with $E \in C^2$ in the interior of $\cS_{j,p}$ for sufficiently large $j$.\footnote{In other words, a part of the
edge curve $\partial B$ contained in $\cS_{j,p}$ can be described as a $C^2$ function $x_1 = E(x_2)$ ( or $x_2 = E(x_1)$).}

Thus, we are facing the following two cases (see also Figure \ref{fig:cases2ab}):
\begin{itemize}
\item {\em Case 2a}. The edge curve $\partial B$ can be parameterized \wql{by either $(E(x_2),x_2)$ or  $(x_1,E(x_1))$ with $E \in C^2$
in the interior of $\cS_{j,p}$ such that, for any $\lambda \in \Lambda_{j,p}$, there exists some $\hat{x} = (\hat{x}_1,\hat{x}_2)
\in \wqq{\intt(\cQ_{j,p}) \cap \intt(\Sp(\psi_{\lambda}))} \cap \partial B$ satisfying either $|E'(\hat{x}_2)| \leq 2 $ or $|E'(\hat{x}_1)|^{-1} \leq 2$.}
\item {\em Case 2b}. The edge curve $\partial B$ can be parameterized \wql{by either $(E(x_2),x_2)$ or  $(x_1,E(x_1))$ with $E \in C^2$
in the interior of $\cS_{j,p}$ such that, for any $\lambda \in \Lambda_{j,p}$, there exists some $\hat{x} = (\hat{x}_1,\hat{x}_2)
\in \wqq{\intt(\cQ_{j,p}) \cap \intt(\Sp(\psi_{\lambda}))} \cap \partial B$ satisfying either $|E'(\hat{x}_2)| > 2 $ or $|E'(\hat{x}_1)|^{-1} > 2$.}
Here, we identify $E'(\hat x_1) = 0$ with $|E'(\hat x_1)|^{-1} = \infty > 2$.
\end{itemize}
\begin{figure}[ht]
\begin{center}
\includegraphics[height = 1in]{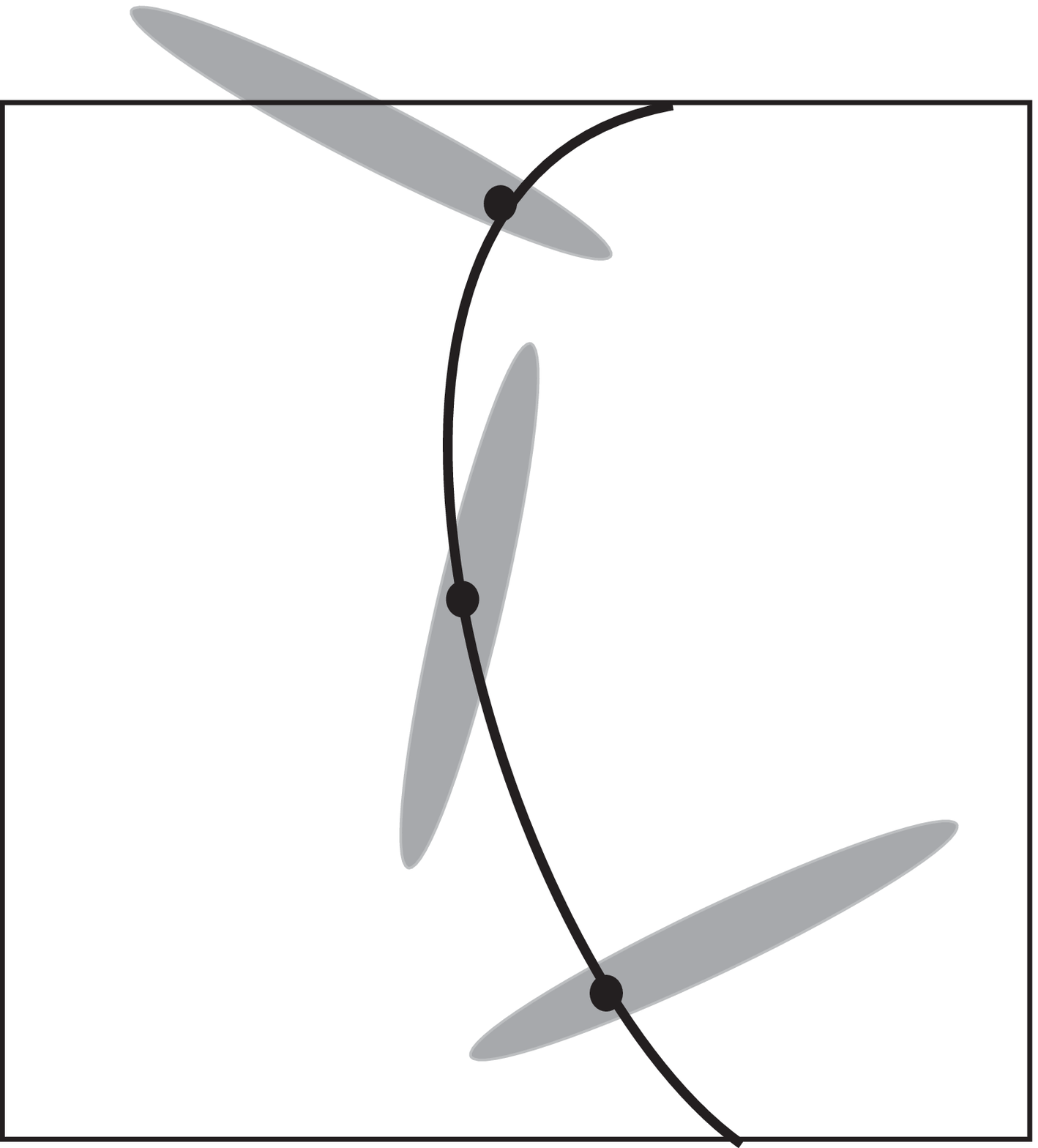}
\put(-40,-15){\footnotesize{(a)}}
\hspace*{4cm}
\includegraphics[height = 1in]{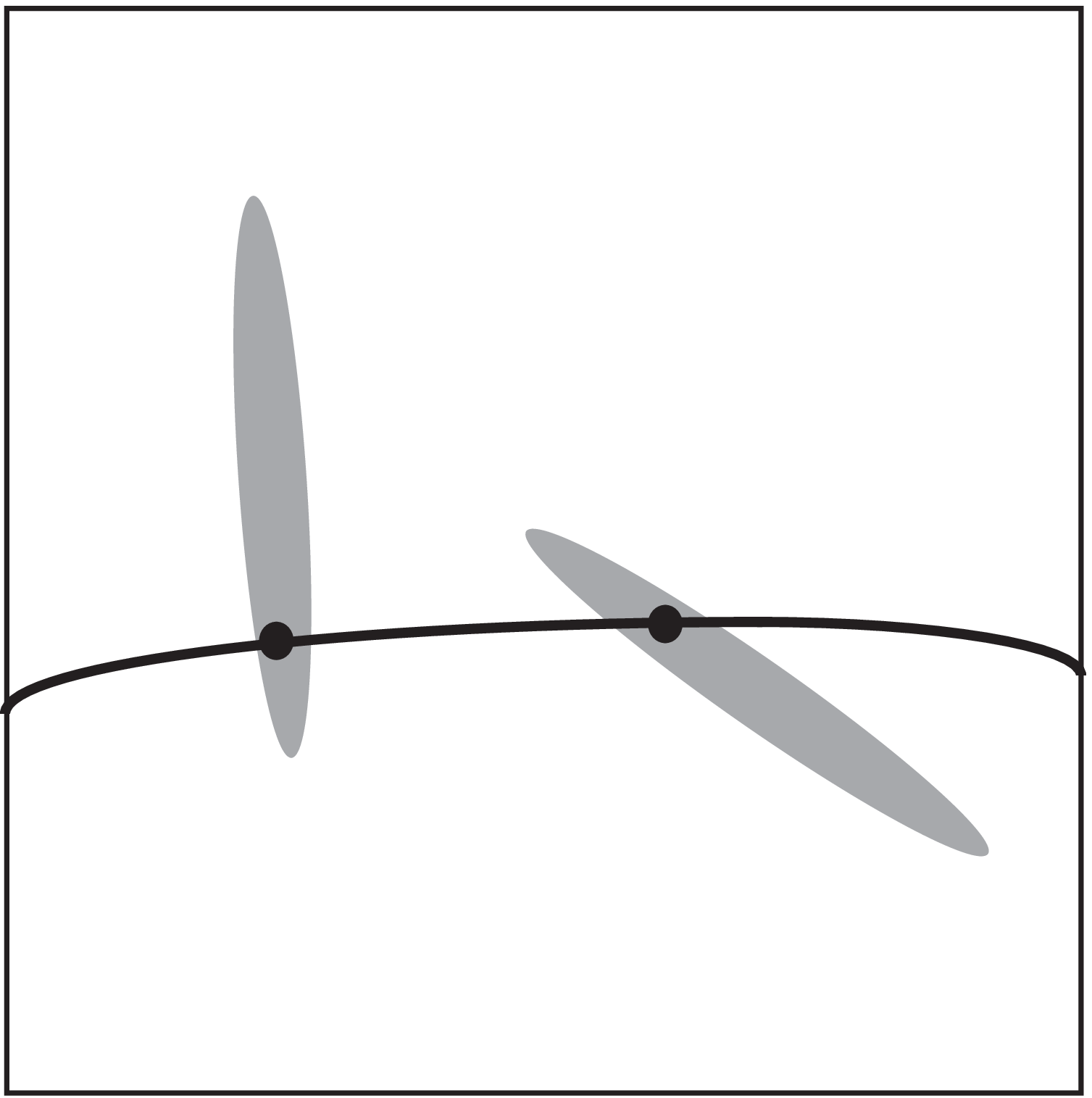}
\put(-40,-15){\footnotesize{(b)}}
\end{center}
\caption{(a) A part of the curve $\partial B$ satisfying Case 2a.
(b) A part of the curve $\partial B$ satisfying Case 2b.}
\label{fig:cases2ab}
\end{figure}
For both cases, we will derive the below stated upper estimates \eqref{eq:estimate1} and \eqref{eq:estimate2} for the absolute value of
 the associated shearlet coefficients. The proofs of these estimates are contained in Section \ref{sec:discontinuity}.
\begin{proposition}\label{prop:main2}
Let $\psi \in L^2(\RR^2)$  be compactly supported, and assume that $\psi$ satisfies conditions (i) and (ii) of Theorem \ref{theo:main}.
Further, let $j \ge 0$ and $p \in \ZZ^2$, and let $\lambda \in \Lambda_{j,p}$. For fixed $\hat{x}=(\hat{x}_1,\hat{x}_2) \in
\wqq{\intt(\cQ_{j,p}) \cap \intt(\Sp(\psi_{\lambda}))} \cap \partial B$,
let $s$ be the slope\footnote{Notice that here we regard the slope of the tangent to a curve
$(E(x_2),x_2)$, i.e., we consider $s$ of a curve indexed by the $x_2$-axis, for instance, by $x_1 = sx_2+b$. For analyzing shearlets
$\tilde{\psi}_{j,k,m}$, the roles of $x_1$ and $x_2$ would need to be reversed.} of the tangent to the edge curve $\partial B$
at $(\hat{x}_1,\hat{x}_2)$, more precisely,
\begin{enumerate}
\item[{\rm (i)}] if $\partial B$ is parameterized by $(E(x_2),x_2)$ with $E \in C^2$ in the interior of $\cS_{j,p}$, then $s = E'(\hat{x}_2)$,
\item[{\rm (ii)}] if $\partial B$ is parameterized by $(x_1,E(x_1))$ with $E \in C^2$ and $E'(\hat{x}_1) \neq 0$ in the interior of $\cS_{j,p}$,
then $s = ({E'(\hat{x}_1)})^{-1}$, and
\item[{\rm (iii)}] if $\partial B$ is parameterized by $(x_1,E(x_1))$ with $E \in C^2$ and $E'(\hat{x}_1) = 0$ in the interior of $\cS_{j,p}$, then
$s = \infty$.
\end{enumerate}
Then there exists some $C > 0$ such that
\begin{equation}\label{eq:estimate1}
|\langle f,\psi_{\lambda}\rangle| \leq C \cdot \frac{2^{-\frac{3}{4}j}}{|\wq{k+}2^{j/2}s|^3},  \qquad \text{if }|s| \leq \wq{3}.
\end{equation}
and
\begin{equation}\label{eq:estimate2}
|\langle f,\psi_{\lambda}\rangle| \leq C \cdot 2^{-\frac{9}{4}j}, \qquad \text{if }|s| > \wq{\frac{3}{2}} \text{ or } s = \infty,
\end{equation}
\end{proposition}

Notice that in Case 2a, condition (i) or (ii) can occur, whereas in Case 2b, all three conditions can occur.


\subsection{Proof of Theorem \ref{theo:main}}

Let $f \in \cE^2(\nu)$. We first observe that, by Proposition \ref{prop:main1}, we can neglect those shearlet coefficients
whose spatial support of the associated shearlet does not intersect the discontinuity curve.

To estimate the remaining shearlet coefficients, we need to analyze their decay properties.
For this, let $j \ge 0$ be sufficiently large and let $p \in \ZZ^2$, be such that the associated cube
satisfies $\cQ_{j,p} \in \cQ_j$. We note that all sets $\wqq{\intt(\Sp{(\psi_{\lambda})})}$ with $\lambda \in \Lambda_{j,p}$
are contained in the interior of ${\mathcal{S}}_{j,p}$. Therefore weights as in \cite{CD04} and  \cite{GL07} are
not required here.

Letting $\eps > 0$, our goal will now be to estimate first $|\Lambda_{j,p}(\eps)|$ and then $|\Lambda(\eps)|$.
WLOG we might assume $\|\psi\|_1 \le 1$, which implies
\[
|\langle f,\psi_{\lambda}\rangle| \le 2^{-3j/4}.
\]
Hence, for estimating $|\Lambda_{j,p}(\eps)|$, it is sufficient to restrict our attention to scales
\beq \label{eq:upperbdonj}
j \le \frac{4}{3}\log_2(\eps^{-1}).
\eeq

We will now deal with Case 2a and Case 2b separately.

\vspace*{0.2cm}

{\em Case 2a}. First, we let $s$ be the slope of the tangent to the edge curve $\partial B$
at $(\hat{x}_1,\hat{x}_2)$ as defined in Proposition \ref{prop:main2}, i.e.,
if  $\partial B$ is parameterized by $(E(x_2),x_2)$ in the interior of $\cS_{j,p}$, then $s = E'(\hat x_2)$,
and if  $\partial B$ is parameterized by $(x_1,E(x_1))$ in the interior of $\cS_{j,p}$, then $s = E'(\hat x_1)^{-1}$.

By assumption of Case 2a, we have $s \in [-2,2]$. Now observe that, for each shear index $k$,
\begin{equation}\label{eq:number}
|\{\lambda=(j,k,m) : \lambda \in \Lambda_{j,p}\}| \le C \cdot (|k+2^{j/2}s|+1).
\end{equation}
Interestingly, this estimate is independent of the choice of the point $\hat x \in \wqq{\intt(\cQ_{j,p}) \cap \intt(\Sp(\psi_{\lambda}))} \cap \partial B$, which
can be seen as follows: Let $\hat x \in \wqq{\intt(\cQ_{j,p}) \cap \intt(\Sp(\psi_{\lambda}))} \cap \partial B, \wang{\hat x' \in \wqq{\intt(\cQ_{j,p}) \cap \intt(\Sp(\psi_{\lambda'}))} \cap \partial B}$, and let $s$ and $s'$ be the
associated slopes of the tangents to the edge curve $E$ in $\hat x$ and $\hat x'$, respectively. Since $E \in C^2$,
\beq \label{eq:ss'}
|s-s'| \le C_1 \cdot 2^{-j/2},
\eeq
and hence
\begin{equation*}
|k+2^{j/2}s'| \leq C \cdot (|k+2^{j/2}s|+1).
\end{equation*}
This proves that the estimate \eqref{eq:number} remains the same asymptotically, independent of the values of $s$ and $s'$,
and hence of $\hat x$ and $\hat x'$.

We further require the following even stronger observation: Fix some $\hat x \in \wqq{\intt(\cQ_{j,p}) \cap \intt(\Sp(\psi_{\lambda}))} \cap \partial B$,
and let $s$ be the associated slope of the tangent to the edge curve $E$ in $\hat x$. Now consider another $\hat x' \in \wqq{\intt(\cQ_{j,p})
\cap \wang{\intt(\Sp(\psi_{\lambda'}))}} \cap \partial B$, and again let $s'$ be the associated slope of the tangent to the edge curve $E$ in $\hat x'$.
Then, for sufficiently \wq{large} scaling index $j$, by \eqref{eq:ss'}, $|s-s'|$ is sufficiently small \wq{and} we may assume $s' \in [-3,3]$.
Hence the estimate \eqref{eq:estimate1} from Proposition \ref{prop:main2} holds not only for $\hat x$ (with $s$), but also for $\hat x'$
(with $s'$). In fact, it can even be checked that by substituting \wq{$s'$} by \wq{$s$} in \eqref{eq:estimate1} the
asymptotic behavior of the estimate for $|\Lambda_{j,p}(\epsilon)|$ does not change. Let us briefly outline the reasoning. First, we observe that
WLOG we can assume that $|k+2^{j/2}s| \ge 2 \cdot C_1$, where $C_1$ is the constant appearing in \eqref{eq:ss'}, since
$|\{k \in \Z : |k+2^{j/2}s| < 2 \cdot C_1\}| \leq C$ with $C$ being independent on $j$ and hence it can be deduced that the hypothesis $|k+2^{j/2}s| \ge 2 \cdot C_1$
for each $j \ge 0$ does not affect our asymptotic estimate of $|\Lambda_{j,p}(\epsilon)|$. From \eqref{eq:ss'}, it then follows that
\[
|k+2^{j/2}s| \leq 2 \cdot |k+2^{j/2}s'|,
\]
which in turn implies
\[
\frac{2^{-\frac{3}{4}j}}{|k+2^{j/2}s'|^3} \leq 8\frac{2^{-\frac{3}{4}j}}{|k+2^{j/2}s|^3}.
\]
Hence, by substituting \wq{$s'$} by \wq{$s$} in \eqref{eq:estimate1} the asymptotic behavior of the estimate for $|\Lambda_{j,p}(\epsilon)|$ does not change.
Thus, it suffices to consider just one fixed $\hat x \in \wqq{\intt(\cQ_{j,p}) \cap \intt(\Sp(\psi_{\lambda}))} \cap \partial B$ with associated slope $s$
in each $\cQ_{j,p}$. We now turn to estimating $|\Lambda_{j,p}(\eps)|$ in this case.

For this, by estimate \eqref{eq:estimate1} from Proposition \ref{prop:main2}, $|\langle f,\psi_{\lambda}\rangle|>\eps$ implies
\begin{equation}\label{eq:number2}
|k+2^{j/2}s| \le C\cdot \eps^{-1/3}\cdot 2^{-j/4}.
\end{equation}
From \eqref{eq:number} and \eqref{eq:number2}, we then conclude
\begin{equation}\label{eq:case1}
|\Lambda_{j,p}(\eps)| \le C \cdot \sum_{k \in K_j(\eps)}(|k+2^{j/2}s|+1) \le C \cdot (\eps^{-1/3}\cdot 2^{-j/4}+1)^2,
\end{equation}
where $K_j(\eps) = \{k\in\bZ:|k+2^{j/2}s|\le C \cdot \eps^{-1/3} \cdot 2^{-j/4}\}$.

\vspace*{0.2cm}

{\em Case 2b}. Exploiting similar arguments as in Case 2a, it also suffices to consider just one fixed $\hat x \in \wqq{\intt(\cQ_{j,p}) \cap
\intt(\Sp(\psi_{\lambda}))} \cap \partial B$ with associated slope $s$ in each \wqq{$\intt(\cQ_{j,p})$}. Again, our goal is now to estimate $|\Lambda_{j,p}(\eps)|$.

By estimate \eqref{eq:estimate2} from Proposition \ref{prop:main2}, $|\langle f,\psi_{\lambda}\rangle|>\eps$ implies
\[
C \cdot 2^{-\frac 94 j} \ge \eps,
\]
hence
\begin{equation}\label{eq:case2}
\quad \text{and} \quad j \leq \frac 49 \log_2{(\eps^{-1})}+C.
\end{equation}
Since there exists some $C$ with
\[
|\Lambda_{j,p}| \le C \cdot 2^j,
\]
it then follows that
\beq\label{eq:estimateLambda_jp}
|\Lambda_{j,p}(\eps)| \le C \cdot 2^{j}.
\eeq
Notice that this last estimate is exceptionally crude, but will be sufficient for our purposes.

\vspace*{0.2cm}

We now combine the estimates for $|\Lambda_{j,p}(\eps)|$ derived in Case 2a and Case 2b. Since
\[
\#(\cQ_j) \leq C\cdot 2^{j/2}.
\]
by \eqref{eq:case1} (and \eqref{eq:upperbdonj}) and by \eqref{eq:case2} (and \eqref{eq:estimateLambda_jp}), we have
\begin{eqnarray}\label{eq:totalnumber} \nonumber
|\Lambda(\eps)| &\leq& C \cdot \Bigl[ \sum_{j=0}^{\frac 43 \log_2(\eps^{-1})} 2^{j/2}\Bigl(\eps^{-1/3}\cdot 2^{-1/4j}+1\Bigr)^2
+ \hspace*{-0.2cm}\sum_{j=0}^{\frac 49 \log_2(\eps^{-1})+C}2^{\frac 32j}\Bigr]\\
&\leq& C \cdot \eps^{-2/3} \cdot \log_2(\eps^{-1}).
\end{eqnarray}

Having estimated $|\Lambda(\eps)|$, we are now ready to prove our main claim. For this, set $N = |\Lambda(\eps)|$, i.e., the total
number of shearlets $\psi_{\lambda}$ such that the magnitude of the corresponding shearlet coefficient
$\langle f,\psi_{\lambda} \rangle$ is larger than $\eps$. By \eqref{eq:totalnumber}, the value $\eps$ can be written as a function
of the total number of coefficients $N$ in the following way:
\[
\eps(N) \le C \cdot N^{-3/2}\cdot (\log N)^{3/2}, \quad \text{for sufficiently large} \,\,N>0.
\]
This implies that
\[
|\theta(f)|_{N} \le C \cdot N^{-3/2}\cdot (\log N)^{3/2}.
\]
Hence,
\[
\sum_{n > N}|\theta(f)|_{n}^2 \le C\cdot N^{-2} \cdot (\log N)^3 \quad \text{for sufficiently large }\,\,N>0,
\]
which proves \eqref{eq:upper}.
The proof of Theorem \ref{theo:main} is finished.


\section{Analysis of Shearlet Coefficients associated with the Smooth Part of a Cartoon-Like Image}
\label{sec:smoothpart}

In this section, we will prove Proposition \ref{prop:main1}. For this, we first prove a
result which shows that, provided that the shearlet $\psi$ satisfies certain decay conditions
even with strong weights  such as $(2^{4j})_j$, the system $\Psi(c;\psi)$ forms a Bessel-like
sequence for $C^2(\RR^2)$ with compact support in $[0,1]^2$.

In the following we will use the notation $r_j \sim s_j$ for $r_j, s_j \in \RR$,
if $C_1 \cdot r_j \le s_j \le C_2 \cdot r_j$ with constants $C_1$ and $C_2$ independent on the
scale $j$.

\begin{lemma}\label{lemma:smoothpart}
Let $\psi \in L^2(\RR^2)$ satisfy
\[
|\hat\psi(\xi)| \le C_1 \cdot \min(1,|\xi_1|^{\alpha}) \cdot \min(1,|\xi_1|^{-\gamma}) \cdot \min(1,|\xi_2|^{-\gamma})
\mbox{ for all }\xi = (\xi_1,\xi_2) \in \RR^2 \hspace*{-0.17cm},
\]
where $\gamma>3$, $\alpha > \gamma+2$, and $C_1$ is some constant. Then, there exists some $C>0$ such that,
for all  $g \in C^2(\RR^2)$ with compact support in $[0,1]^2$,
\[
\sum_{j=0}^{\infty}\sum_{|k| \le \lceil 2^{j/2} \rceil}\sum_{m \in \bZ^2} 2^{4j}|\langle \wqqq{g},\psi_{j,k,m}\rangle|^2 \leq C\cdot
\wqqq{\left\|\frac{\partial^2}{\partial x_1^2} g\right\|_2^2}.
\]
\end{lemma}

The proof of this lemma will explore the following result from \cite{KKL10a}, which we state here
for the convenience of the reader.

\begin{proposition}\cite{KKL10a}\label{proposition:upperbd}
Let $\psi \in L^2(\RR^2)$ satisfy
\[
|\hat\psi(\xi)| \le C_1 \cdot \min(1,|\xi_1|^{\alpha}) \cdot \min(1,|\xi_1|^{-\gamma}) \cdot \min(1,|\xi_2|^{-\gamma})
\mbox{ for all }\xi = (\xi_1,\xi_2) \in \RR^2 \hspace*{-0.17cm},
\]
where $\alpha>\gamma>3$ and $C_1$ is some constant. Then, there exists some $C>0$ such that, for all $\eta \in L^2(\RR^2)$,
\[
\sum_{j=0}^{\infty}\sum_{|k| \le \lceil 2^{j/2}\rceil}\sum_{m \in \bZ^2} |\langle \eta,\psi_{j,k,m}\rangle|^2 \leq C\cdot \|\eta\|_2^2.
\]
\end{proposition}

\begin{proof}{\em \hspace*{-0.2cm} (Proof of Lemma \ref{lemma:smoothpart})}.
By the assumption on $\psi$, the parameters $\alpha$ and $\gamma$ can be chosen
such that
\[
|\hat\psi(\xi)| \le C_1 \cdot \min(1,|\xi_1|^{\alpha}) \cdot \min(1,|\xi_1|^{-\gamma}) \cdot \min(1,|\xi_2|^{-\gamma})
\mbox{ for all }\xi = (\xi_1,\xi_2) \in \RR^2 \hspace*{-0.17cm},
\]
where $\alpha > \gamma+2, \gamma > 3$. Now, let $\eta \in L^2(\RR^2)$ be chosen to satisfy
\[
\frac{\partial^2}{\partial x_1^2} \eta = \psi.
\]
Then a straightforward computation shows that $\eta$ satisfies the hypotheses of Proposition \ref{proposition:upperbd}.
Using integration by parts,
\wqqq{
\[
\left|\left\langle \frac{\partial^2}{\partial x_1^2}g,\eta_{j,k,m}\right\rangle\right|^2 = 2^{4j}|\langle g,\psi_{j,k,m}\rangle|^2,
\]
}
hence, by Proposition \ref{proposition:upperbd},
\begin{eqnarray*}
\sum_{j=0}^{\infty}\sum_{|k| \le \lceil 2^{j/2} \rceil}\sum_{m \in \bZ^2} 2^{4j} |\langle g,\psi_{j,k,m}\rangle|^2
&=& \sum_{j=0}^{\infty}\sum_{|k| \le \lceil 2^{j/2} \rceil}\sum_{m \in \bZ^2} \left|\left\langle \frac{\partial^2}{\partial x_1^2}g,
\eta_{j,k,m}\right\rangle\right|^2
\\
&<& C \cdot\left  \|  \frac{\partial^2}{\partial x_1^2} g \right \|_2^2.
\end{eqnarray*}
The proof is complete.
\end{proof}

This now enables us to derive Proposition \ref{prop:main1} as a corollary.

\begin{proof}{\em \hspace*{-0.2cm} (Proof of Proposition \ref{prop:main1})}.
Set
\[
\tilde{\Lambda}_j = \{\lambda \in \Lambda_j : \Sp(\psi_{\lambda}) \cap \Sp(g) \neq \emptyset \},\quad j > 0,
\]
i.e., $\tilde{\Lambda}_j$ is the set of indices in $\Lambda_j$ associated with shearlets
whose support intersects the support of $g$. Then, for each $J > 0$, we have
\beq \label{eq:NJ}
N_{J} = \Big|\bigcup_{j=0}^{J-1} \tilde{\Lambda}_j \Big| \sim 2^{2J}.
\eeq
Now, first observe that there exists some $C>0$ such that
\begin{eqnarray*}
\sum_{j = 1}^{\infty}2^{4j} \sum_{n>N_{j}} |\theta(g)|_n^2
& \leq & C \cdot \sum_{j=1}^{\infty}\sum_{j'=j}^{\infty}\sum_{k,m} 2^{4j} |\langle \wqqq{g},\psi_{j',k,m}\rangle|^2\\
& = & C\cdot\sum_{j'=1}^{\infty}\sum_{k,m}|\langle \wqqq{g},\psi_{j',k,m}\rangle|^2\left( \sum_{j=1}^{j'}2^{4j}\right).
\end{eqnarray*}
By Lemma \ref{lemma:smoothpart}, this implies
\[
\sum_{j = 1}^{\infty}2^{4j} \sum_{n>N_{j}} |\theta(g)|_n^2
\leq C \cdot \sum_{j'=1}^{\infty}\sum_{k,m}2^{4j'} |\langle \wqqq{g},\psi_{j',k,m}\rangle|^2 < \infty
\]
and hence, also by \eqref{eq:NJ},
\[
\sum_{n>N_{j}} |\theta(g)|_n^2 \leq C \cdot (2^{2j})^{-2} \leq \wq{C} \cdot N_{j}^{-2}.
\]
Finally, let $N>0$. Then there exists a positive integer $j_0>0$ satisfying
\[
N \sim N_{j_0} \sim 2^{2j_0},
\]
and the claim is proved.
\end{proof}


\section{Analysis of Shearlet Coefficients associated with the Discontinuity Curve}
\label{sec:discontinuity}

\subsection{Proof of Proposition \ref{prop:main2}}
Let $(j,k,m) \in \Lambda_{j,p}$, and fix $\hat{x}=(\hat{x}_1,\hat{x}_2) \in \wqq{\intt(\cQ_{j,p}) \cap \intt(\Sp(\psi_{\lambda}))} \cap \partial B$.
Let $s$ be the slope of the tangent to the edge curve $\partial B$
at $(\hat{x}_1,\hat{x}_2)$ as defined in Proposition \ref{prop:main2}, i.e.,
if  $\partial B$ is parameterized by $(E(x_2),x_2)$ in the interior of $\cS_{j,p}$, then $s = E'(\hat x_2)$,
and if  $\partial B$ is parameterized by $(x_1,E(x_1))$ in the interior of $\cS_{j,p}$, then $s = E'(\hat x_1)^{-1}$,
where we now assume that $E'(\hat x_1) \neq 0$ and consider the case $E'(\hat x_1) = 0$ later.

By translation symmetry, WLOG we can assume
that the edge curve satisfies $E(0)=0$ with $(\hat{x}_1,\hat{x}_2) = (0,0)$.
Further, since the conditions (i) and (ii) in Theorem \ref{theo:main}
are independent on the translation parameter $m$, it does not play a role in our analysis. Hence, WLOG we choose $m = 0$.
Also, since $\psi$ is compactly supported, there exists some $L > 0$ such that $\Sp{(\psi)} \subset [-L/2,L/2]^2$. By a rescaling
argument, we can might \wq{assume} $L=1$. Even more, WLOG we can assume that $\Sp{(\psi)} = [-1/2,1/2]^2$, which implies
\begin{eqnarray} \nonumber
\Sp(\psi_{j,k,0})
& = & \{ x \in \RR^2 : -2^{-j/2} k x_2 -2^{-j-1}\leq x_1 \leq -2^{-j/2} k x_2+2^{-j-1},\\ \label{eq:supp}
& &  -2^{-\frac{j}{2}} \leq 2x_2 \leq 2^{-\frac{j}{2}}\},
\end{eqnarray}
since this will not change our asymptotic estimate for $|\Lambda_{j,p}(\epsilon)|$.

Let now $f \in \cE^2(\nu)$, and select ${\mathcal{P}}$ to be the smallest parallelogram which entirely contains the discontinuity curve
parameterized by $(E(x_2),x_2)$ or $(x_1,E(x_1))$ in the interior of $\Sp(\psi_{j,k,0})$ and \wq{whose two sides are parallel to
the tangent $x_1 = sx_2$ to the discontinuity curve at $(\hat{x}_1,\hat{x}_2) = (0,0)$.} For an illustration, we refer to Figure \ref{fig:edgecurve1}.
We now split the coefficients $|\langle f,\psi_{j,k,0}\rangle|$ into the part `inside the parallelogram' and `outside' of it
exploiting the shearing property of shearlets for the second part, and obtain
\begin{equation}\label{eq:shear_coeff}
|\langle f,\psi_{j,k,0}\rangle| = \wq{|\langle \chi_{\mathcal{P}} f,\psi_{j,k,0}\rangle|}
+ |\langle \chi_{{\mathcal{P}}^c}f(S_s \cdot),\psi_{j,\hat k,0}\rangle|
\end{equation}
where $\hat k = k+2^{j/2}s$. From now on, we assume that $\hat k < 0$ with $\hat k = k+2^{j/2}s$. The case $k+2^{j/2}s \ge 0$ can be handled similarly.

Let us start by estimating the first term \wq{$|\langle \chi_{\mathcal{P}}f,\psi_{j,k,0}\rangle|$} in \eqref{eq:shear_coeff}
stated as

{\bf Claim 1.}
\beq \label{eq:term1}
\wq{|\langle \chi_{\mathcal{P}}f,\psi_{j,k,0}\rangle|} \leq C \cdot \frac{(1+|s|^2)^{3/2}}{2^{3j/4} \cdot |\hat{k}|^{3}}.
\eeq

\begin{figure}[h]
\begin{center}
\hspace{-50pt}
\includegraphics[width=13cm]{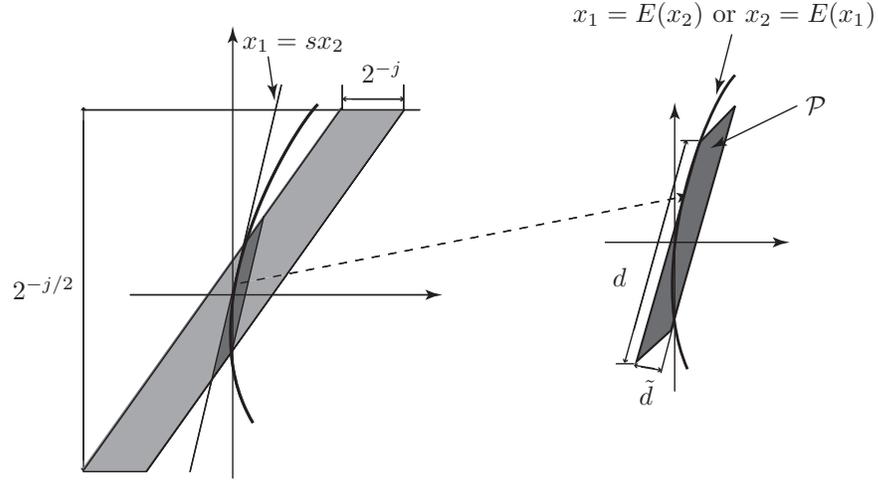}
\put(-85,310){\footnotesize{$x_1=E(x_2)$ or $x_2 = E(x_1)$}}
\put(-297,205){\footnotesize{$ 2^{-j/2}$}}
\put(3,275){\footnotesize{$\mathcal{P}$}}
\put(-165,287){\footnotesize{$2^{-j}$}}
\put(-210,300){\footnotesize{$x_1 = sx_2$}}
\put(-70,210){\footnotesize{$d$}}
\put(-60,165){\footnotesize{$\tilde{d}$}}
\end{center}
\vspace{-130pt}
\caption{A shearlet $\psi_{j,{k},0}$ intersecting the edge curve $x_1 = E(x_2)$ or \wq{$x_2=E(x_1)$.} The right hand side shows
a magnification of the parallelogram \wql{ $\mathcal{P}$}.}
\label{fig:edgecurve1}
\end{figure}

First, notice that the tangent to the edge curve $(E(x_2),x_2)$ (or $(x_1,E(x_1))$) is given by $x_1 = sx_2$. We now assume that the edge curve is
contained in a set $\{x \in \R^2 : x_1 \ge sx_2 \}$, and just remark that the general case can be handled similarly.
Let now $d$ be the length of that side of $\mathcal{P}$, which is a part of the tangent $x_1 = sx_2$. We observe that $d$ is in fact the
distance between two points, in which the tangent $x_1=sx_2$ intersects the boundary of $\Sp (\psi_{j,k,0})$. For an illustration we
wish to refer to Figure \ref{fig:edgecurve1}. From this observation, it follows that
\[
d = \frac{2^{-j/2}\sqrt{1+|s|^2}}{|\hat{k}|}.
\]
We now let $\tilde{d}$ be the height of $\mathcal{P}$. Since the edge curve can be parameterized by a $C^2$ function $E$ with bounded curvature,
\[
\tilde{d} \leq C \cdot \Bigl(\frac{2^{-j/2}\sqrt{1+|s|^2}}{|\hat{k}|}\Bigr)^2.
\]
Summarizing, the volume of \wq{$\mathcal{P}$} can be estimated as
\[
|\mathcal{P}| \le C \cdot \frac{(1+|s|^2)^{3/2}}{2^{3j/2} \cdot {|\hat{k}|^3 }}.
\]
This implies
\[
\wq{|\langle f \chi_{\mathcal{P}},\psi_{j,k,0}\rangle|}
\le C \cdot 2^{3j/4} \cdot \|f\|_{\infty}\cdot \|\psi\|_{\infty} \cdot \frac{(1+|s|^2)^{3/2}}{2^{3j/2} \cdot {|\hat{k}|^3}}
\leq C \cdot \frac{(1+|s|^2)^{3/2}}{2^{3j/4} \cdot |\hat{k}|^{3}},
\]
and Claim 1, i.e., estimate \eqref{eq:term1}, is proved.

Next, we estimate the second term, i.e., $|\langle \chi_{\mathcal{P}^c}f(S_s \cdot),\psi_{j,\hat k,0}\rangle|$, in \eqref{eq:shear_coeff} stated as

{\bf Claim 2.}
\beq \label{eq:term2}
|\langle \chi_{\mathcal{P}^c}f(S_s \cdot),\psi_{j,\hat k,0}\rangle|
\leq C \cdot (1+|s|)^2 \cdot \left( \frac{1}{2^{3j/4} \cdot |\hat{k}|^{3}}+\frac{1}{2^{7j/4} \cdot |\hat{k}|^2}\right).
\eeq

Notice that $S_s^{-1}\mathcal{P}$ entirely contains the edge curve of $f(S_s \cdot)$ in the interior of $\Sp(\psi_{j,\hat k,0})$ and that the boundary of
the parallelogram $S_s^{-1}\mathcal{P}$ consists of two vertical line segments in the interior of $\Sp (\psi_{j,\hat k,0})$ (see Figure \ref{fig:edgecurve1_1}).
By translation symmetry, this implies that for proving Claim 2, it suffices to estimate
\[
\langle f_0(S_s\cdot)\chi_{\Omega},\psi_{j,\hat{k},0}\rangle,
\]
where $\Omega = \{(x_1,x_2) \in \bR^2:x_1>0\}$, $f_0 \in C^2(\RR^2)$ compactly supported in $[0,1]^2$
and $\sum_{|\alpha| \le 2}\|D^{\alpha}f_0\|_{\infty} \le 1$. We wish to mention that the consideration of the case $x_1>0$ -- compare the definition
of the set $\Omega$ -- is by no means restrictive, since the case  $x_1 \leq 0$ can be handled in a similar way.
\begin{figure}[h]
\begin{center}
\hspace{-50pt}
\includegraphics[width=13cm]{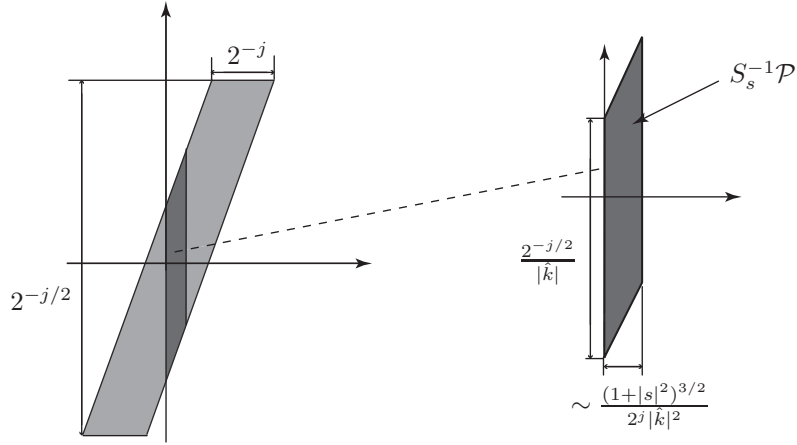}
\put(-277,184){\footnotesize{$ 2^{-j/2}$}}
\put(-5,270){\footnotesize{$S^{-1}_s\mathcal{P}$}}
\put(-195,277){\footnotesize{$2^{-j}$}}
\put(-85,200){\footnotesize{$\frac{2^{-j/2}}{|\hat{k}|}$}}
\put(-65,145){\footnotesize{$\sim \frac{(1+|s|^2)^{3/2}}{2^j|\hat k|^2}$}}
\end{center}
\vspace{-130pt}
\caption{A shearlet $\psi_{j,\hat{k},0}$ intersecting \wq{the parallelogram \wql{ $S_s^{-1}\mathcal{P}$}}. The right hand side shows
a magnification of the parallelogram \wql{ $S_s^{-1}\mathcal{P}$}. }
\label{fig:edgecurve1_1}
\end{figure}

Now, again by translation symmetry, we may assume that $\partial (\supp \psi_{j,\hat{k},0})$ intersects the origin.
In particular, we may translate $\Sp (\psi_{j,\hat k,0})$ defined in \eqref{eq:supp} so that we might now consider
\[
\Sp (\psi_{j,\hat k,0}) + (-2^{j-1},0)\]
as the support of $\psi_{j,\hat k, 0}$. We refer the reader to Figure \ref{fig:edgecurve2} for an illustration.
This implies that there is one side of the boundary $\partial (\supp \psi_{j,\hat{k},0})$, which is a part of the line
\[
\mathcal{L} = \{(x_1,x_2) \in \RR^2 :  x_2 = (-2^{j/2}/\hat{k}) \cdot x_1\}
\]
with slope $-2^{j/2}/\hat{k}$, as described in Figure \ref{fig:edgecurve2}. Applying the Taylor expansion for \wq{$f_0(S_s \cdot)$} at
each point $x = (x_1,x_2) \in \mathcal{L}$, we obtain
\[
\wq{ f_0(S_s x) }= a(x_1)+b(x_1)\left(x_2+\frac{2^{j/2}}{\hat{k}} \cdot x_1\right)+c(x_1,x_2)\left(x_2+\frac{2^{j/2}}{\hat{k}}\cdot x_1\right)^2,
\]
where $a(x_1),b(x_1)$ and $c(x_1,x_2)$ are all bounded in absolute value by $C(1+|s|)^2$. This implies (compare also an illustration
of the area of integration in Figure \ref{fig:edgecurve2})
{\allowdisplaybreaks
\begin{eqnarray}\label{eq:splitting0}
|\langle \wql{ \wq{f_0(S_s \cdot)\chi_{\Omega}}},\psi_{j,\hat{k},0}\rangle|
&=& \left| \int_{0}^{-\frac{\hat{k}}{2^{j}}}
\int_{-\frac{2^{j/2}}{\hat{k}}\cdot x_1}^{-\frac{2^{j/2}}{\hat{k}}\cdot x_1- \frac{2^{-j/2}}{\hat{k}}}
\wq{f_0(S_s x)} \psi_{j,\hat{k},0}(x) \, dx_2dx_1\right| \nonumber \\
&\le& C\cdot \wql{ (1+|s|)^2} \cdot \left| \int_{0}^{-\frac{\hat{k}}{2^{j}}}\wq{\sum_{\ell=1}^{3}I_{\ell}(x_1)} \, dx_1 \right|,
\end{eqnarray}
}
where
{\allowdisplaybreaks
\begin{eqnarray*}
I_1(x_1) &=& \left| \int_{0}^{- \frac{2^{-j/2}}{\hat{k}}}T_{\beta}\left(\psi_{j,\hat{k},0}(x_1,x_2)\right)dx_2\right|, \\
I_2(x_2) &=& \left| \int_{0}^{- \frac{2^{-j/2}}{\hat{k}}}x_2 \cdot T_{\beta}\left(\psi_{j,\hat{k},0}(x_1,x_2)\right)dx_2\right|, \\
I_3(x_2) &=& \left| \int_{0}^{- \frac{2^{-j/2}}{\hat{k}}}x_2^2 \cdot T_{\beta}\left(\psi_{j,\hat{k},0}(x_1,x_2)\right)dx_2\right|,
\end{eqnarray*}
}
with $T_{\beta}$ being the translation operator defined by $T_{\beta}(f) = f(\cdot -\beta)$ and $\beta \in \RR$ being chosen
to be $\beta = (0,(2^{j/2}/\hat{k}) \cdot x_1).$

\begin{figure}[ht]
\begin{center}
\includegraphics[width=5.5cm]{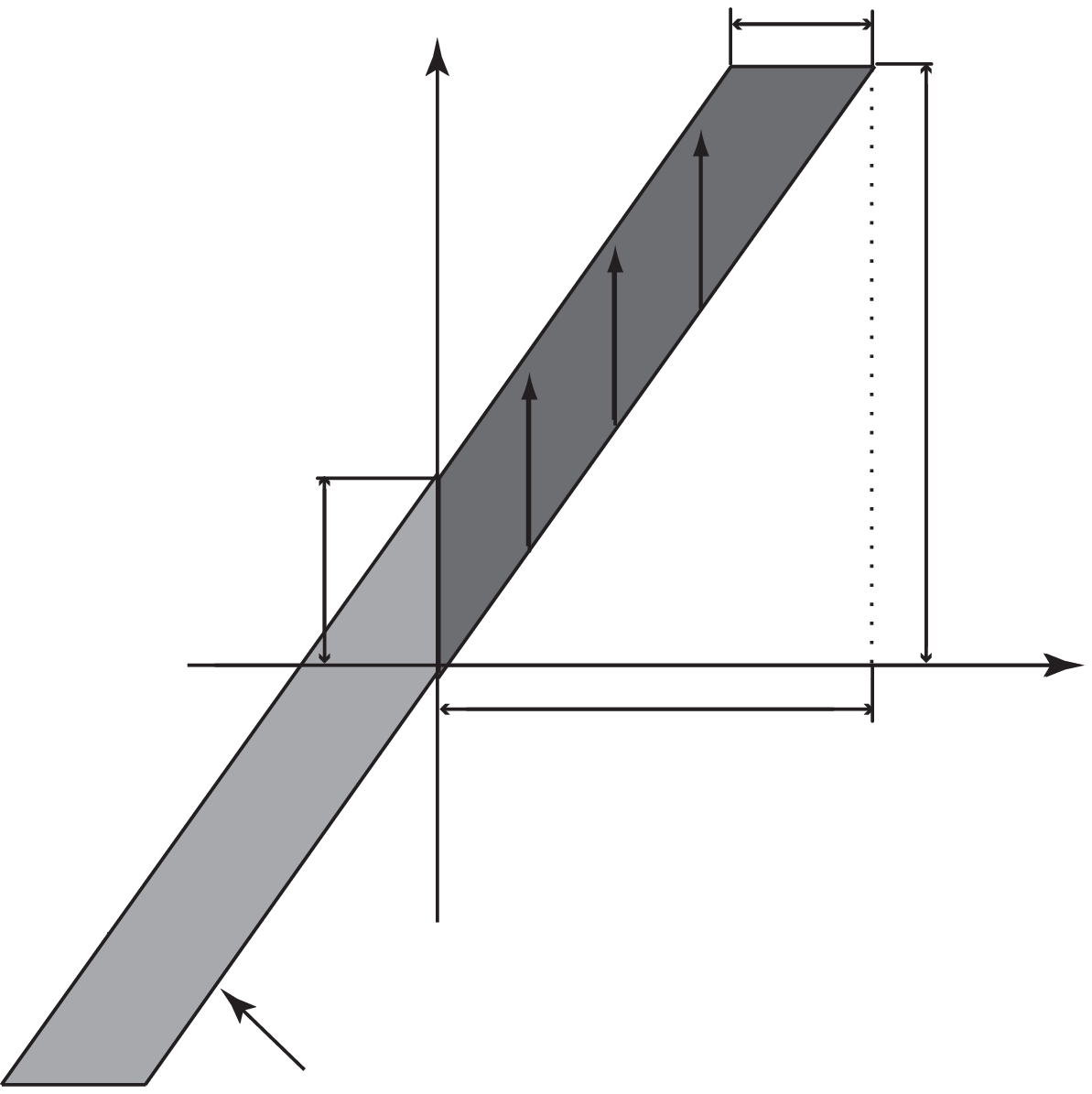}
\put(-142,78){\footnotesize{$\frac{2^{-j/2}}{|\hat{k}|}$}}
\put(-70,42){\footnotesize{$\frac{|\hat{k}|}{2^j}$}}
\put(-18,105){\footnotesize{$ 2^{-j/2}$}}
\put(-110,0){\footnotesize{$\mathcal{L}$}}
\put(-46,160){\footnotesize{$ 2^{-j}$}}
\end{center}
\caption{\wang{A shearlet $\psi_{j,\hat{k},0}$ intersecting the edge curve $x_1 = 0$  such
that $\supp(\psi_{j,\hat{k},0})$ intersects the positive $x_2$ axis and
$\partial (\supp \psi_{j,\hat{k},0})$ intersects the origin. The illustration also highlights the integration area for \eqref{eq:splitting0}. }}
\label{fig:edgecurve2}
\end{figure}

We first estimate $I_1(x_1)$. We observe that, since
\[
\{(x_1,x_2) \in \bR^2 : \psi_{j,\hat{k},0}(x_1,x_2) \ne 0\} \subset [0,2^{-j/2}/|\hat{k}|] \quad \text{for a fixed } x_1>0,
\]
WLOG, for any $x_1 > 0$, the interval $[0, 2^{-j/2}/\wqq{|\hat{k}|}]$
for the range of the integration in $I_1(x_1)$ can be replaced by $\bR$ (see also Figure \ref{fig:edgecurve2}).
Therefore, we have
\beq\label{eq:estimateI}
I_1(x_1) = \left|\int_{\bR} \psi_{j,\hat{k},0}(x_1,x_2)dx_2\right|
= \left| \int_{\bR}  \hat{\psi}_{j,\hat{k},0}(\xi_1,0) \cdot e^{2\pi i \wql{ x_1}\xi_1}d\xi_1\right|.
\eeq
Now
\[
|\hat{\psi}_{j,\hat{k},0}(\xi_1,\xi_2)| = 2^{-3j/4} \cdot |\hat\psi(2^{-j}\xi_1,2^{-j/2}\xi_2-2^{-j}\hat{k}\xi_1)|.
\]
and hence, by hypothesis (i) from Theorem \ref{theo:main},
\begin{equation}\label{eq:fourier1}
|\hat{\psi}_{j,\hat{k},0}(\xi_1,0)| \leq 2^{j/4} \cdot |2^{-j}h(2^{-j}\xi_1)| \cdot |\hat{k}|^{-\gamma}.
\end{equation}
By \eqref{eq:estimateI} and \eqref{eq:fourier1}, it follows that
\begin{equation}\label{eq:mainestimate1}
I_1(x_1) \le C \cdot \frac{2^{j/4}}{|\hat{k}|^{\gamma}} \quad \mbox{for some } C > 0.
\end{equation}

Next, we estimate $I_2(x_1)$. We have
\begin{eqnarray*}
I_2(x_1)
& \leq & \left|\int_{\bR} x_2 \cdot \psi_{j,\hat{k},0}(x_1,x_2)dx_2\right|
+ \frac{2^{j/2}}{|\hat{k}|} \cdot |x_1| \cdot \left|\int_{\bR}\psi_{j,\hat{k},0}(x_1,x_2)dx_2 \right|\\
& = & S_1 + S_2.
\end{eqnarray*}
To estimate $S_1$, observe that, by hypothesis (ii) from Theorem \ref{theo:main},
\begin{eqnarray}\nonumber
S_1 & = & \wqqq{\frac{1}{2\pi}}\left| \int_{\bR}
\left(\frac{\partial}{\partial \xi_2} \hat{\psi}_{j,\hat{k},0}\right)(\xi_1,0) e^{2\pi i \xi_1 x_1}d\xi_1\right|\\ \label{eq:estimateS1}
& \le & \frac{1}{2\pi}\int_{\bR} (2^{-\frac{j}{4}} \cdot h(2^{-j}\xi_1)) \cdot 2^{-j} \cdot |\hat{k}|^{-\gamma} d\xi_1.
\end{eqnarray}
By \eqref{eq:mainestimate1} and the fact that $0 \le x_1 \le \frac{|\hat{k}|}{2^j}$, the second
term $S_2$ can be estimated as
\beq \label{eq:estimateS2}
S_2 \le C \cdot \left(\frac{2^{j/2}}{|\hat{k}|}|x_1|\right) \cdot \frac{2^{j/4}}{|\hat{k}|^{\gamma}}
\le C \cdot \left( 2^{-j/2}\right) \cdot \frac{2^{j/4}}{|\hat{k}|^{\gamma}}
\le \frac{C}{2^{j/4} \cdot |\hat{k}|^{\gamma}}.
\eeq
Concluding from \eqref{eq:estimateS1} and \eqref{eq:estimateS2}, we obtain
\begin{equation}\label{eq:mainestimate2}
I_2(x_1) \le S_1 + S_2  \le \frac{C}{2^{j/4} \cdot |\hat{k}|^{\gamma}}.
\end{equation}

Finally, we estimate $I_3(x_1)$. For this, notice that $2^{-3j/4}T_{\beta}(\psi_{j,\hat{k},0}(x_1,x_2))$ is bounded, hence
\begin{equation}\label{eq:mainestimate3}
I_3(x_1) \le 2^{\frac{3}{4}j} \cdot C \cdot \left|\int_{0}^{\frac{-1}{2^{j/2}\hat{k}}}x_2^2 \, dx_2\right| \leq \frac{C}{2^{\frac{3}{4}j} \cdot |\hat{k}|^{3}}.
\end{equation}

Summarizing, by \eqref{eq:splitting0}, \eqref{eq:mainestimate1}, \eqref{eq:mainestimate2}, and \eqref{eq:mainestimate3},
\begin{eqnarray*}
|\langle \wq{f_0(S_s \cdot)\chi_{\Omega}},\psi_{j,\hat{k},0}\rangle|
& \le & C \cdot (1+|s|)^2 \cdot \int_{0}^{\frac{|\hat{k}|}{2^j}}\left(\frac{2^{\frac{1}{4}j}}{|\hat{k}|^{\gamma}} \nonumber
+\frac{1}{2^{\frac{3}{4}j} \cdot |\hat{k}|^3}\right)dx_1\\
& \le & C \cdot (1+|s|)^2 \cdot \left( \frac{1}{2^{3j/4} \cdot |\hat{k}|^{3}}+\frac{1}{2^{\frac{7}{4}j} \cdot |\hat{k}|^2}\right),
\end{eqnarray*}
and Claim 2, i.e., estimate \eqref{eq:term2}, is proved.

From Claim 1 and 2, i.e., from \eqref{eq:term1} and \eqref{eq:term2}, we conclude that
\[
|\langle f,\psi_{j,\wq{{k}},0}\rangle| \leq C\Bigl[ (1+|s|)^2\Bigl( \frac{1}{2^{\frac{3}{4}j} \cdot |\hat k|^3}+\frac{1}{2^{\frac{7}{4}j} \cdot |\hat k|^2}\Bigr)
+(1+|s|^2)^{3/2}\frac{1}{2^{\frac{3}{4}j} \cdot |\hat k|^3}\Bigr].
\]
This implies \eqref{eq:estimate1} and \eqref{eq:estimate2} except for the case $s = \infty$, which we will study now.

Finally, we consider the case $s = \infty$, i.e., we assume that the edge curve is parameterized by $(x_1,E(x_1))$ in the interior of $\cS_{j,p}$
such that  $E \in C^2$ and $E'(\hat x_1) = 0$. As before, let $f \in \cE^2(\nu)$, and select ${\mathcal{P}}$ to be the smallest parallelogram which
entirely contains the discontinuity curve in the interior of $\Sp(\psi_{j,k,0})$ \wq{and whose two sides are parallel to $x_2 = 0$}. Similarly as before, we consider
\begin{equation} \label{eq:shear_coeff2}
|\langle f,\psi_{j,k,0}\rangle| = |\langle \chi_{\mathcal{P}}f,\psi_{j,k,0}\rangle| + |\langle \chi_{{\mathcal{P}}^c}f,\psi_{j,k,0}\rangle|.
\end{equation}
and estimate both terms on the RHS separately.

We first consider the term $|\langle \chi_{\mathcal{P}}f,\psi_{j,k,0}\rangle|$. Using similar arguments as before, which we decided not to
include in detail to avoid repetitions, one can prove that
\begin{equation}\label{eq:t1}
|\langle \chi_{\mathcal{P}}f,\psi_{j,k,0}\rangle| \leq C \cdot 2^{-\frac{9}{4}j}.
\end{equation}

\wq{Turning our attention} to the second term $|\langle \chi_{{\mathcal{P}}^c}f,\psi_{j,k,0}\rangle|$, we first observe that
\wq{$\mathcal{P}$ entirely contains the edge curve of $f$ in the interior of $\Sp(\psi_{j, k,0})$ and that the boundary of
the parallelogram $\mathcal{P}$ consists of two horizontal line segments in the interior of $\Sp (\psi_{j,k,0})$ (see Figure \ref{fig:edgecurve3}).
By translation symmetry, this implies that for the second term in \eqref{eq:shear_coeff2}, i.e., for $|\langle \chi_{{\mathcal{P}}^c}f,\psi_{j,k,0}\rangle|$,
it suffices to estimate} $|\langle f_0\chi_{\tilde{\Omega}},\psi_{j,k,0}\rangle|$, where $\tilde{\Omega} = \{(x_1,x_2) \in \bR^2:x_2>0\}$ and
$f_0 \in C^2(\RR^2)$ compactly supported in $[0,1]^2$ with $\sum_{|\alpha| \le 2}\|D^{\alpha}f_0\|_{\infty} \le 1$.
As before, the consideration of the case $x_2>0$ -- compare the definition of the set $\tilde{\Omega}$ -- is by no means restrictive, since the case
$x_2 \leq 0$ can be handled in a similar way. Observe that \wq{the (horizontal) vanishing moment condition}
follows from condition (i) in Theorem \ref{theo:main}. Let us briefly think about this: Letting $\xi_2$ be fixed, condition (i)
immediately implies $\hat \psi(0,\xi_2) = 0$. Then, by applying Taylor expansion, it follows that
$\hat \psi(\xi_1,\xi_2) = \frac{\partial\hat \psi}{\partial \xi_1}(0,\xi_2) \cdot \xi_1 + \mathcal{O}(\xi_1^2)$, and,
again by condition (i), we have $\frac{\hat \psi(\xi_1,\xi_2)}{\xi_1} \to 0$ as $\xi_1 \rightarrow 0$, hence
$\frac{\partial\hat \psi}{\partial \xi_1}(0,\xi_2) = 0$. This procedure can now be continued.
Especially, condition (i) from Theorem \ref{theo:main} implies
\begin{equation}\label{eq:moment}
\int_{\bR} x_1^{\ell}\cdot\psi(x_1,x_2)dx_1 = 0 \quad \mbox{for all } x_2 \in \bR \mbox{ and } \ell=0,1.
\end{equation}
Further, we utilize that the shearing operation $S_k$ preserves vanishing moments along the $x_1$ axis.
That is,
\wq{
\[
\int_{\bR}x_1^{\ell} \cdot \psi(S_k (x_1,x_2)^T)dx_1 = 0
\quad \mbox{for all }k, x_2 \in \bR \,\, \mbox{and}\,\, \ell=0,1.
\]
}
This can be seen as follows. For $\ell=0, 1$ and for a fixed $x_2 \in \bR$, the function
$x_1 \mapsto (x_1-kx_2)^{\ell}$ is a polynomial
of degree less than or equal to $\ell$, hence, by condition \eqref{eq:moment} on the number of
vanishing moments on $\psi$, we have
\[
\int_{\bR}x_1^{\ell} \cdot \psi(S_k (x_1,x_2)^T)dx_1 = \int_{\bR}(x_1-kx_2)^{\ell} \cdot \psi(x_1,x_2)dx_1 = 0
\quad \mbox{for all }k \in \bR.
\]
\begin{figure}[ht]
\begin{center}
\includegraphics[width=5.5cm]{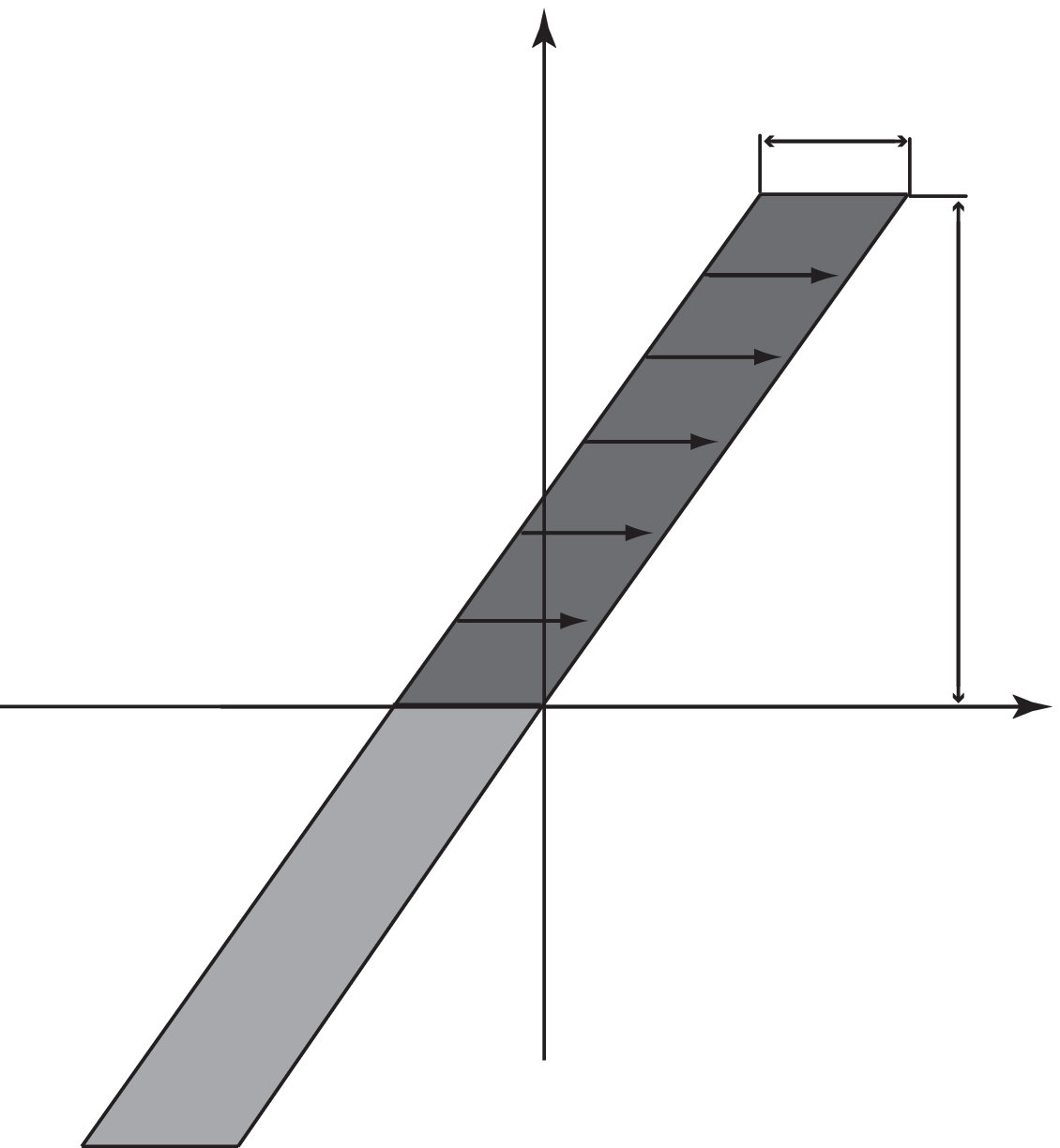}
\put(-38,155){\footnotesize{$2^{-j}$}}
\put(-9,105){\footnotesize{$2^{-j/2}$}}
\end{center}
\caption{\wang{A shearlet $\psi_{j,\hat{k},0}$ intersecting the edge curve $x_2 = 0$ such
that $\supp(\psi_{j,\hat{k},0})$ intersects the positive $x_1$ axis and
$\partial (\supp \psi_{j,\hat{k},0})$ intersects the origin. The illustration also highlights the integration area for \eqref{eq:t2}. }}
\label{fig:edgecurve3}
\end{figure}

Employing Taylor expansion and integration (compare Figure \ref{fig:edgecurve3}) similar to the
proof of the previous case, we finally obtain
\begin{equation}\label{eq:t2}
|\langle f_0\chi_{\tilde{\Omega}},\psi_{j,\wq{{k}},0} \rangle|
\le C \cdot 2^{3j/4} \cdot \int_{0}^{\cdot 2^{-j/2}}\int_{-2^{-j}}^{0} \, x_1^2 dx_1dx_2
\le C \cdot 2^{-11j/4}
\end{equation}
\wql{
By \eqref{eq:t1} and \eqref{eq:t2}, we obtain \eqref{eq:estimate2} when $s = \infty$. This completes the proof. }

\end{document}